\documentclass[12pt,a4paper]{amsart}
\usepackage[english]{babel}
\usepackage{ucs}
\usepackage[utf8x]{inputenc}
\usepackage{graphicx,psfrag}
\usepackage{moreverb}
\usepackage{amsmath,amssymb,amsfonts,amsthm,latexsym}
\usepackage{cases}
\usepackage{moreverb} 
\usepackage{float}	
\usepackage{enumitem}
\usepackage{tikz}
\usetikzlibrary{calc}

\newtheorem{theorem}{Theorem}

\newtheorem{proposition}{Proposition}
\theoremstyle{definition}
\newtheorem*{definition}{Definition}

\newcommand{\bR}{\mathbb{R}}

\newcommand{\bZ}{\mathbb{Z}}
\newcommand{\bQ}{\mathbb{Q}}
\newcommand{\cC}{\mathcal{C}}
\newcommand{\cK}{\mathcal{K}}
\newcommand{\cB}{\mathcal{B}}
\newcommand{\uL}{\mathbf{L}}
\newcommand{\uP}{\mathbf{P}}
\newcommand{\ux}{\mathbf{x}}
\newcommand{\uy}{\mathbf{y}}

\newcommand{\uv}{\mathbf{v}}

 \title[Parametric geometry of numbers]
{On a Conjecture of Schmidt for the Parametric Geometry of Numbers}

\author{Aminata Keita}

\address{
   Department of Mathematics\\
   University of Ottawa\\
   585 King Edward\\
   Ottawa, Ontario K1N 6N5, Canada}
  \email{akeit104@uottawa.ca}
  

\thanks{Work partially supported by NSERC (Canada)}

  \date{ \today}
  
\begin{document} 
\baselineskip=20pt

\begin{abstract}
With the help of the recently introduced parametric geometry of numbers by W.~M.~Schmidt and L.~Summerer, we prove a strong version of a conjecture of  Schmidt concerning the successive minima of a lattice.
\end{abstract}
\maketitle

\section{Introduction}
Among the conjectures proposed by W.~M.~Schmidt in 1983, one is concerned with the parametric geometry of numbers \cite[Conjecture~2]{Sc}. This conjecture was proven in 2012 by N.~G.~Moshchevitin  \cite[Theorem~1]{Mg}. The goal of this paper is to prove a stronger statetement along the same lines and we will show that this generalization is the best possible. We start by recalling Moshchevitin's result, using the notations of D.~Roy in \cite{Rd}.

Fix an integer $n\geqslant2$. For each non-zero $\xi\in\bR^{n+1}$, we associate the family of convex bodies
\begin{equation*}
\label{c1}
 \mathcal{C}_{\xi}(Q) :=  \left\lbrace \ux \in \bR^{n+1}\,;\, \Vert \ux \Vert \leqslant 1\,,\,|\ux\cdot \xi| \leqslant Q^{-1} \right\rbrace
 \qquad (Q\geqslant 1),
\end{equation*} 
where $\ux\cdot \uy$ denotes the standard scalar product in $\bR^{n}$ and $||\ux|| = (\ux\cdot \ux)^{1/2}$ denotes the euclidean norm of $\ux$. Define
\begin{equation*}
L_{\xi,j}(q)= \log \lambda_{j}\left(\cC_{\xi}(e^{q}); \bZ^{n+1}\right) \qquad \quad (q\geqslant 0;\,\, 1\leqslant j \leqslant n+1) ,
\end{equation*}
where $\lambda_{j}(\cC;\Lambda)$ is defined for a convex body $\cC$ and lattice $\Lambda$ in $\bR^{n+1}$ to be the $j$-th minimum of $\cC$ with respect to $\Lambda$, i.e.\ the smallest $\lambda\geqslant0$ such that $\lambda\cC$ contains at least $j$ linearly independent elements of $\Lambda$. Clearly, we have
\begin{equation*}
\label{m4}
  L_{\xi,1}(q)\leqslant \cdots \leqslant L_{\xi,n+1}(q)  \qquad (q\geqslant 0).
\end{equation*}

The functions $ L_{\xi,j} :[0,\infty)\longrightarrow \bR$ $(1\leqslant j \leqslant n+1 )$ are continuous and piecewise linear, with slopes alternating between $0$ and $1$ (see \cite[$\mathsection 2$]{Rd}, \cite[$\mathsection 3$]{S2013a}). Moreover, since the volume of $\cC_{\xi}(e^{q})$ is bounded below and above by multiples of $e^{-q}$, Minkowski's theorem implies that  
\[
q- \sum\limits_{j=1}^{n+1} L_{\xi,j}(q)
\]
is a bounded function in $q$, and so the average of the $L_{\xi,j}$'s is $q/(n+1)$. If the coordinates of $\xi$ are linearly independent over $\bQ$, then for each $j=1,\ldots,n+1$, there exists arbitrarily large values of $q$ such that
\[
L_{\xi,j}(q)=L_{\xi,j+1}(q)
\]
(see \cite[Theorem~1]{S2009}). On the other hand, we have the following result.

\begin{theorem}[N.~G.~Moshchevitin, $2012$]
\label{th1}
For each integer $k$ with $2\leqslant k\leqslant n$, there exists $\xi\in\bR^{n+1}$ whose coordinates are linearly independent over $\bQ$ such that
\begin{equation*}
\lim_{q\rightarrow \infty}\left( L_{\xi,k-1}(q)-\frac{q}{n+1}\right) = -\infty \quad \text{and} \quad \lim_{q\rightarrow \infty} 
\left( L_{\xi,k+1}(q)-\frac{q}{n+1}\right) = \infty.
\end{equation*} 
\end{theorem}

Thus, the functions $L_{\xi,k-1}(q)$ and $L_{\xi,k+1}(q)$ can diverge from each other on each side by $q/(n+1)$. Our main result improves upon these estimates, and is stated as follows.

\begin{theorem}
\label{th2}
For each integer $k$ with $2\leqslant k\leqslant n$, there exist uncountably many vectors $\xi\in\bR^{n+1}$ whose coordinates are linearly independent over $\bQ$ such that
\begin{equation*}
\label{ee}
\lim_{q\rightarrow \infty} \frac{L_{\xi,k-1}(q)}{q} =0 \quad \text{and} \quad  \liminf_{q\rightarrow \infty} \frac{L_{\xi,k+1}(q)}{q} = \frac{1}{n-k+2}.
\end{equation*}
\end{theorem}

Further, this result is the best possible in the following sense.

\begin{theorem}
\label{th3}
Let $k$ be an integer with $2\leqslant k\leqslant n$, and suppose that $\xi$ is a point in $\bR^{n+1}$ whose coordinates are linearly independent over $\bQ$ and which satisfies  
\[
\underset{q\rightarrow\infty}{\lim} \frac{L_{\xi,k-1}(q)}{q} =0.
\]
Then, we have
\[
\underset{q\rightarrow\infty}{\liminf} \frac{L_{\xi,k+1}(q)}{q}\leqslant \frac{1}{n-k+2}.
\]
\end{theorem}

In the following section, we state Schmidt's original conjecture, and we justify the above reformulation of Moshchevitin's result. In section 3, we use the results of \cite[$\mathsection 4$]{Rp} to prove Theorem \ref{th2}. Finally, section 4 provides a proof of theorem \ref{th3} by using Schmidt and Summerer's parametric geometry of numbers.

\section{Link with Schmidt's Original Conjecture}
For each $N\in\bR$ with $N\geqslant1$ and for each $\xi=(1, \xi_{1},\ldots, \xi_{n})\in\bR^{n+1}$, Schmidt \cite{Sc} introduced the lattice  $\Lambda(\xi,N)\subset\bR^{n+1}$  generated by the vectors 
\[
\uv_{0}= (N^{-1},N^{1/n}\xi_{1},\ldots,N^{1/n}\xi_{n}), \,\,
\uv_{1}= (0,-N^{1/n},\,\ldots \,,0),\,\,
\ldots ,\,\,
\uv_{n}= (0,\,0,\,\ldots,\,-N^{1/n}),
\]
and defined
\[
 \mu_{j}(\xi, N)=\lambda_{j}(\cB;\Lambda(\xi,N)) \qquad \quad (1\leqslant j \leqslant n+1 )
\]
where
$
\cB= \left\lbrace (y_{0},y_{1},\ldots,y_{n})\in \bR^{n+1}; \,\,|y_{i}|\leqslant 1,\, i=0,\ldots,n\right\rbrace 
$
is the unit hypercube in $\bR^{n+1}$. 

With these notations, he conjectured the following result, later proven by Moshchevitin.

\begin{theorem}[N.~G.~Moshchevitin, $2012$]
\label{th4}
Let $k$ be an integer with $2\leqslant k\leqslant n$. There exists real numbers $\xi_{1},\ldots, \xi_{n} \in [0,1)$ such that
\begin{itemize}
\item[$\bullet$]
$1, \xi_{1},\ldots, \xi_{n}$ are linearly independent over $\bQ$;
\item[$\bullet$]
$\underset{N\rightarrow\infty}{\lim}\mu_{k-1}(\xi,N)=0\quad\text{and}\quad
\underset{N\rightarrow\infty}{\lim}\mu_{k+1}(\xi, N)=\infty$, where $\xi=(1,\xi_{1},\ldots, \xi_{n})$.
\end{itemize}
\end{theorem}

In fact, Schmidt's original conjecture omits the linear independence condition, but as Moshchevitin mentions in his article, (see \cite[$\mathsection 3$]{Mg}), the conjecture is trivial without this hypothesis.

To show the equivalence between Theorems \ref{th1} and \ref{th4}, fix a point $\xi=(1,\xi_{1},\ldots, \xi_{n})\in\bR^{n+1}$ whose coordinates are linearly independent over $\bQ$, and fix an integer $k$ with $2\leqslant k\leqslant n$. In \cite[$\mathsection 1$]{Mg}, Moshchevitin begins by observing that
\begin{equation*}
\label{m1}
\mu_{j}(\xi, N)= \lambda_{j}(\cK_{\xi}(N);\bZ^{n+1}) \qquad (N\geqslant 1,\,\, 1\leqslant j \leqslant n+1 ),
\end{equation*}
where
\[
\cK_{\xi}(N)=\left\lbrace (x_{0},\ldots,x_{n}) \in \bR^{n+1}\,;\,|x_{0}| \leqslant N \,,\,|x_{0}\xi_{j} -x_{j}| 
\leqslant N^{-1/n},\,\, j=1,\ldots,n \right\rbrace.
\]
Consequently, the second statement of theorem \ref{th4} can be rewritten as 
\begin{equation}
\label{e4}
\lim_{N \rightarrow \infty}\lambda_{k-1}(\cK_{\xi}(N);\bZ^{n+1})= 0 \quad \text{and} \quad \lim_{N \rightarrow \infty}
\lambda_{k+1}(\cK_{\xi}(N);\bZ^{n+1})=\infty.
\end{equation}

Meanwhile, Mahler's duality theorem yields 
\begin{equation*}
\lambda_{j}(\cK_{\xi}(N);\bZ^{n+1})\lambda_{n-j+2}(\cK_{\xi}^{\star}(N);\bZ^{n+1})\asymp 1 \qquad (1 \leqslant j\leqslant n+1),
\end{equation*}
where
\[
\cK_{\xi}^{\star}(N)= \left\lbrace \ux  \in \bR^{n+1}\,;\,|\ux\cdot \xi| \leqslant N^{-1} \,,\,||\ux|| \leqslant N^{1/n} \right\rbrace
\]
is essentially the convex body dual to $\cK_{\xi}(N)$.
Thus, the conditions in (\ref{e4}) become
\begin{equation}
\label{e5}
\lim_{N \rightarrow \infty}\lambda_{n+3-k}(\cK_{\xi}^{\star}(N);\bZ^{n+1})= \infty \quad \text{and} \quad \lim_{N \rightarrow \infty}
\lambda_{n+1-k}(\cK_{\xi}^{\star}(N);\bZ^{n+1})= 0.
\end{equation} 
On the other hand, since
$\cC_{\xi}(e^{q})=  e^{-q/(n+1)} \cK_{\xi}^{\star}(e^{nq/(n+1)})$, it follows that
\begin{equation*}
L_{\xi,j}(q)=\frac{q}{n+1} + \log\lambda_{j}(\cK_{\xi}^{\star}(e^{nq/(n+1)});\bZ^{n+1}) \qquad (1 \leqslant j\leqslant n+1) .
\end{equation*} 
Thus, the conditions in (\ref{e5}) can be rewritten as
\begin{equation*}
 \lim_{q\rightarrow \infty} \left( L_{\xi,n+3-k}(q)-\frac{q}{n+1}\right) = \infty \quad \text{and} \quad  \lim_{q\rightarrow \infty}
 \left( L_{\xi,n+1-k}(q)-\frac{q}{n+1} \right) = -\infty .
\end{equation*} 
The equivalence between theorems \ref{th1} and \ref{th4} follows.

\section{Proof of the Main Result}

In order to prove Theorem \ref{th2}, we need to etablish some prelimanary results which relie on the following basic construction.
\begin{proposition}
\label{pr1}
 Let $a,b,c,\alpha, \beta, \gamma \in (0,\infty)$ with $a<b<c$. There exists a unique choice of real numbers $r,s,t,u \in (0,\infty)$ with $r<s<t<u$ and a unique triplet of continuous and piecewise linear functions $(A,B,C)$ on $[r,u]$ such that the union of their graphs is as in Figure 1, i.e.
\begin{enumerate}[label=\roman{enumi}$)$]
\item
for all $q \in [r, u]$, we have
\begin{equation}
\label{p4}
 A(q)\leqslant B(q)\leqslant C(q)
\quad
 \text{and}
\quad
\frac{1}{\alpha} A(q)+\frac{1}{\beta} B(q)+\frac{1}{\gamma} C(q)=q ;
\end{equation}
 \item
the function $A$ is constant equal to $a$ on $[r, t]$, has slope $\alpha$ on $[t, u]$, and satisfies $A(u)=b $;
\item
the function $B$ has slope $\beta$ on $[r, s]$, is constant equal to $b$ on $[s, u]$, and satisfies $B(r)=a$;
\item
the function $C$ is constant equal to $b$ on $[r, s]$, has slope $\gamma$ on $[s, t]$, and is constant equal to $c$ on $[t, u]$.
 \end{enumerate}
 
 \begin{figure}[h!]
\begin{tikzpicture}
\draw[thick,domain=-4:4] plot (\x, {1});
\draw[ thick,domain=-4:8] plot (\x, {3});
\draw[  thick, domain=4:8] plot (\x, {5});
\draw[  thick, domain=-4:-2] plot (\x, \x+5);
\draw[  thick, domain=-2:4] plot (\x, 0.333*\x + 3.667);
\draw[  thick, domain=4:8] plot (\x, 0.5*\x-1);

\draw[fill] (-4,1) circle [radius=2pt];
\draw[fill] (-2,3) circle [radius=2pt];
\draw[fill] (4,5) circle [radius=2pt];
\draw[fill] (4,1) circle [radius=2pt];
\draw[fill] (8,3) circle [radius=2pt];

\draw [ thick, dashed] (-2,3)--(-2,0)node[below]{$s$ };
\draw [ thick, dashed] (4,5)--(4,0)node[below]{$t$ };
\draw [ thick, dashed] (8,5)--(8,0)node[below]{$u$ };
\draw [ thick, dashed] (-4,1)--(-4,0) node[below]{$r$ };

\draw[fill] (-4,1)      node[left]{ $a $};
\draw[fill] (-4,3)      node[left]{ $b $};
\draw[fill] (8,3)      node[right]{ $b$};
\draw[fill] (8,5)      node[right]{ $c$};

\draw[fill] (-2.9,2)      node[right]{ $B$};
\draw[fill] (2,3)      node[above]{ $B $};
\draw[fill] (1,1)      node[above]{\ $A $};
\draw[fill] (-3,3)      node[above]{  $C $};
\draw[fill] (6,5)      node[above]{\ $C $};

\draw[fill] (0.4,4)      node[left]{ $\gamma $};
\draw[fill] (-3.1,2)      node[left]{ $\beta$};
\draw[fill] (5.5,1.5)      node[right]{ $\alpha $};

\end{tikzpicture}
\caption{ }
\end{figure}

\end{proposition}
\begin{proof}
If there exist real numbers $r,s,t,u$ and functions $A,B,C$ as in the claim, then substituting $q$ by $r,s,t,u$ in the second condition of (\ref{p4}) yields, respectively, 
 \begin{align}
 \label{p5}
  r&= \frac{a}{\alpha} +\frac{a}{\beta} +\frac{b}{\gamma}   ; \qquad 
  s= \frac{a}{\alpha}  +\frac{b}{\beta} +\frac{b}{\gamma}   ; \qquad
  t= \frac{a}{\alpha}  +\frac{b}{\beta} +\frac{c}{\gamma}  ; \qquad 
  u= \frac{b}{\alpha} +\frac{b}{\beta} +\frac{c}{\gamma}  ,
 \end{align}
which uniquely determines them all.

Now, let $r,s,t,u$ be given by (\ref{p5}). Since $r<s<t<u$, there exists a unique triplet 
of continuous functions $(A,B,C)$  on $[r,u]$ with constant slopes on $[r,s]$, $[s,t]$ and $[t,u]$, and with
\begin{align*}
A(r)=A(s)=A(t)=a \quad & \text{and} \quad A(u)=b ,\\
B(r)=a \quad &\text{and} \quad B(s)=B(t)=B(u)=b,\\ 
C(r)=C(s)=b \quad &\text{and} \quad C(t)=C(u)=c.
\end{align*}
Thus, the function $F= \frac{1}{\alpha} A+\frac{1}{\beta} B+\frac{1}{\gamma} C$ is continuous and of constant slope on each of the interval $[r,s]$, $[s,t]$, and $[t,u]$. By construction, we have that  $F(q)=q$ for $q=r,s,t,u$. Thus,
\[
F(q)= q  \quad \text{for all $q\in[r,u]$}.
\]
Since $A$ and $C$ are constant on $[r,s]$, this implies that $B$ has slope $\beta$ on $[r,s]$. 
Similarly, we deduce that $C$ has slope $\gamma$ on $[s,t]$, and that $A$ has slope $\alpha$ on $[t,u]$.
\end{proof}

\begin{proposition}
\label{pr2}
With the same notation as above, suppose that $b/a<c/b$. Then, we have 
\begin{align}
\underset{\tiny q\in[r , u]}{\max}\frac{A(q)}{q}&= \frac{a}{r} \qquad \text{and} \qquad
\underset{\tiny q\in[r , u]}{\min}\frac{C(q)}{q}= \frac{b}{s} .
\end{align}
\end{proposition}

\begin{proof}
First, using (\ref{p5}) note that
\begin{equation*}
\label{p6}
\frac{a}{t}< \frac{b}{u}<\frac{a}{r}  \quad \text{and} \quad \frac{b}{s}<\frac{b}{r}<\frac{c}{u}<\frac{c}{t}.
\end{equation*}
Since $a/r < \alpha$ and $b/s<\gamma$, it follows that the ratio $A(q)/q$ is decreasing on $[r,t]$ and increasing on $[t,u]$, and that the ratio $C(q)/q$ is decreasing on $[r,s]$, increasing on $[s,t]$ and decreasing on $[t,u]$. The conclusion follows straightforwardly.
\end{proof}

Let $\Delta$ denote the set of sequences $(a_{m})_{m\in \bZ}$ of positive reals which satisfy
 \begin{align*}
\quad \quad \quad 1<  \frac{a_{m+1}}{a_{m}}< \frac{a_{m+2}}{a_{m+1}} \qquad (m\in \bZ),\\
  & \\
 \underset{m\rightarrow-\infty}{\lim}a_{m}=0\quad \text{and} \quad \underset{m \rightarrow\infty}{\lim}
\frac{a_{m+1}}{a_{m}}=+\infty.
\end{align*} 
The following result further extends the preceding propositions.

\begin{proposition}
\label{pr3}
Let $(a_{m})_{m\in \bZ}\in \Delta $ and let $\alpha, \beta, \gamma \in (0,\infty)$. Define 
\begin{equation}
\label{r1}
r_{m}=\frac{a_{m}}{\alpha} +\frac{a_{m}}{\beta} +\frac{a_{m+1}}{\gamma} \qquad (m\in \bZ).
\end{equation}
Then, there exists a unique triplet of continuous and piecewise linear functions $(A,B,C)$ on $(0,\infty)$ whose restriction to the interval $[r_{m}, r_{m+1}]$ fulfills the conditions of Proposition \ref{pr1} with $a=a_{m}$, $b=a_{m+1}$ and $c=a_{m+2}$ for each $m\in\bZ$. Moreover, we have 
\begin{equation}
\label{p8}
\underset{q\rightarrow \infty}{\lim} A(q)= \infty , \quad
\underset{q\rightarrow \infty}{\limsup}\,\frac{A(q)}{q}=0 \quad \text{and} \quad
\underset{q\rightarrow \infty}{\liminf} \,\frac{C(q)}{q}=\frac{\beta\gamma}{\beta +\gamma}.
\end{equation}
\end{proposition}

\begin{proof}
Let $(a_{m})_{m\in \bZ}\in \Delta$, and define
\begin{equation}
\label{k1}
s_{m}=\frac{a_{m}}{\alpha} +\frac{a_{m+1}}{\beta} +\frac{a_{m+1}}{\gamma}  \quad \text{and} \quad
 t_{m}=\frac{a_{m}}{\alpha} +\frac{a_{m+1}}{\beta} +\frac{a_{m+2}}{\gamma} \qquad (m\in \bZ).
\end{equation}
By setting $a=a_{m}$, $b=a_{m+1}$ and $c=a_{m+2}$, Proposition \ref{pr1} and (\ref{p5}) yield for each $m\in\bZ$ a triplet of continuous and piecewise linear functions $(A^{(m)},B^{(m)},C^{(m)})$ on $[r,u]= [r_{m}, r_{m+1}]$.

Since the triplets $(A^{(m-1)},B^{(m-1)},C^{(m-1)})$ and $(A^{(m)},B^{(m)},C^{(m)})$
coincide at the point $r_m$  and are equal to $(a_m , a_m, a_{m+1})$ for each $m\in \bZ$, it follows that the sequence of triplets of functions $(A^{(m)},B^{(m)},C^{(m)})$ with $m\in \bZ$ determine a unique triplet of continuous and piecewise linear functions $(A,B,C)$ on $\underset{m \in \bZ}{\cup}[r_{m}, r_{m+1}]=(0,\infty)$. Now, Proposition \ref{pr2} gives 
\begin{equation}
\label{r3}\underset{\tiny q\in[r_{m} , r_{m+1}]}{\max}\small \frac{A(q)}{q}=
\frac{a_{m}}{r_m}
\quad \text{and} \quad
\underset{\tiny q\in[r_{m} , r_{m+1}]}{\min} \small \frac{C(q)}{q}=\small
\frac{a_{m+1}}{s_m},
\end{equation}
and so
\begin{align*}
\limsup_{q\rightarrow \infty}\frac{A(q)}{q}&=  \underset{m\rightarrow\infty}{\lim}
\frac{a_{m}}{r_m}=0 
 \quad \text{and} \quad
\liminf_{q\rightarrow \infty} \frac{C(q)}{q}=\underset{m\rightarrow\infty}{\lim}
\frac{a_{m+1}}{s_m}=\frac{\beta\gamma}{\beta +\gamma}.
\end{align*}
\end{proof}

Our next result uses the notion of \emph{generalized $(n+1)$-system} introduced by D.~Roy in \cite{Rp}. It provides a good approximation of the functions  $\uL_{\xi}$ for non-zero point $\xi \in \bR^{n+1}$ (see \cite{Rp} for more details). We recall here the definition .
\begin{definition}
Let $I$ be a subinterval of $ [0,\infty)$ with non-empty interior. A \emph{generalized $(n+1)$-system} on $I$ is a map $ \uP =(P_{1},\ldots,P_{n+1}): I\longrightarrow \bR^{n+1}$ with the following properties.
\begin{enumerate}
\item[(G1)]
For each $q \in I$, we have  $0\leqslant P_{1}(q)\leqslant\cdots \leqslant P_{n+1}(q)$ and $P_{1}(q)+\cdots + P_{n+1}(q)=q$.
\item[$\text{(G2)}$]
If  $H$ is a non-empty open subinterval of $I$ on which $\uP$ is differentiable, then there are integers 
$\underline{r},\overline{r}$ with $1\leqslant \underline{r} \leqslant \overline{r} \leqslant n+1$ such that $P_{\underline{r}},P_{\underline{r}+1},\ldots,P_{\overline{r}}$ coincide on the whole interval
$H$ and have slope $1/(\overline{r}-\underline{r}+1)$ while any other component $P_{j}$ of $P$ is constant on $H$.
\item[(G3)]
If $q$ is an interior point of $I$ at which $ \uP $ is not differentiable, 
if $\underline{r},\overline{r}$, $\underline{s},\overline{s}$ are the integers for which 
\begin{equation}
\label{e8}
P'_{j}(q^{-})=\frac{1}{\overline{r}-\underline{r}+1} \quad (\underline{r}\leqslant j\leqslant \overline{r}) 
\quad \text{et} \quad P'_{j}(q^{+})=\frac{1}{\overline{s}-\underline{s}+1} \quad (\underline{s}\leqslant j\leqslant \overline{s}) 
\end{equation}
and if  $\underline{r}\leqslant \overline{s}$, then we have $P_{\underline{r}}(q)=P_{\underline{r}+1}(q)=\cdots=P_{\overline{s}}(q)$.
\end{enumerate}
\end{definition} 

We now combine the previous Propositions to etablish the following result.

\begin{proposition}
\label{pr4}
Let $k$ be an integer with $2\leqslant k\leqslant n$. With the notation of Proposition \ref{pr3}, suppose that $\alpha= 1/(k-1)$, $\beta=1$ and $\gamma=1/(n+1-k)$. For all $q>0$, let
\[
P_{1}(q)=\cdots=P_{k-1}(q)=A(q),\quad P_{k}(q)=B(q) \quad \text{and} \quad P_{k+1}(q)=\cdots=P_{n+1}(q).
\]
Then the function $\uP: (0,\infty) \longrightarrow \bR^{n+1}$ defined by
\[
\uP(q):=(P_{1}(q),\ldots,P_{n+1}(q)) \qquad (q> 0)
\]
is an \emph{generalized $(n+1)$-system} on $(0, \infty)$. Moreover, we have 
\begin{equation*}
\underset{q\rightarrow \infty}{\lim} P_{1}(q)= \infty , \quad
\limsup_{q\rightarrow \infty}\frac{P_{k-1}(q)}{q}=0 \quad \text{and} \quad \liminf_{q\rightarrow 
\infty}\frac{P_{k+1}(q)}{q}= 
\frac{1}{n-k+2} .
\end{equation*}
\end{proposition}

\begin{proof}
The components $P_{1},\ldots,P_{n+1}$ of $\uP$ are continuous and piecewise linear on $(0, \infty)$. They satisfy
\[
0\leqslant P_{1}(q)\leqslant\cdots \leqslant P_{n+1}(q) \quad \text{and} \quad P_{1}(q)+\cdots + P_{n+1}(q)=q  \qquad (q > 0).
\]
The function $\uP$ is differentiable on $(0, \infty)$ except at the points $r_{m},s_{m},t_{m}$ given by (\ref{r1}) and (\ref{k1}). On each of the interval $[r_{m}, s_{m}]$, $[s_{m},t_{m}]$, $[t_{m},r_{m+1}]$, the components $P_{1},\ldots,P_{n+1}$ are constant except for few, say $h$ of them, which coincide on the interval and which have slope $1/h$. At the point $r_{m}$, the slopes of $P_{1},\ldots,P_{k-1}$ go from $1/(k-1)$ to $0$, while the slope of $P_{k} $ goes from $0$ to $1$, and all these functions take the same value, i.e. 
\[
P_{1}(r_{m})=\cdots=P_{k}(r_{m}) \qquad (m\in \bZ).
\]
At the point $s_m$, the function $P_{k} $ goes from slope $1$ to slope $0$, while the slopes of $P_{k+1},\ldots,P_{n+1} $ go from $0$ to $1/(n-k+1)$ , and similary.
\[
P_{k}(s_{m})=P_{k+1}(s_{m})=\cdots=P_{n+1}(s_{m}) \qquad (m\in \bZ).
\]
Finally, at the point $t_{m}$, the slopes of  $P_{k+1},\ldots,P_{n+1}$ go from $1/(n-k+1)$  to  $0$, while the slopes of $P_{1},\ldots,P_{k-1} $  go from $0$ to  $1/(k-1)$, and we have  
\[
P_{1}(t_{m})=\cdots= P_{k-1}(t_{m})< P_{k}(t_{m})<P_{k+1}(t_{m}) =\cdots =P_{n}(t_m) \qquad (m\in \bZ).
\]
Therefore, the function $\uP$ is an \emph{generalized $(n+1)$-system} on $(0, \infty)$. The second assertion of the proposition follows from (\ref{p8}).
\end{proof}

In \cite[$\mathsection 4$]{Rp}, D.~Roy shows that for each \emph{generalized $(n+1)$-system} $\uP$ on $[q_{0}, \infty)$ with $q_{0}\geqslant 0$, there exists a non-zero point $\xi$ of $\bR^{n+1}$ such that the difference $\uL_{\xi}-\uP$ is bounded. Then, we have 
\[
\limsup_{q\rightarrow \infty} \frac{L_{\xi,j}(q)}{q}= \limsup_{q\rightarrow \infty}\frac{P_{j}(q)}{q}
\quad \text{and} \quad
\liminf_{q\rightarrow \infty}\frac{L_{\xi,j}(q)}{q}=\liminf_{q\rightarrow \infty}\frac{P_{j}(q)}{q} 
\qquad (1\leqslant j \leqslant n+1).
\]
In the context of Proposition \ref{pr4}, this guarantees the existence of a point $\xi \in \bR^{n+1}$ with 
\begin{align}
\label{m1}
\limsup_{q\rightarrow \infty}\frac{L_{\xi,k-1}(q)}{q}&=0
\quad \text{and} \quad 
\liminf_{q\rightarrow \infty}\frac{L_{\xi,k+1}(q)}{q}= \frac{1}{n-k+2}.
\end{align}
Moreover, since $\underset{q\rightarrow \infty}{\lim} P_{1}(q)= \infty $, the function $L_{\xi,1}$ 
is unbounded. It follows that $\xi$ is a point whose coordinates are linearly independent over $\bQ$. 

To finish the proof, it remains to show that one can construct uncountably many such points. For each $\theta \in (0,\infty)$, we define 
\[
a^{(\theta)}_{m}=\theta2^{m^{3}} \qquad (m\in \bZ).
\]
Then, the sequence $\big(a^{(\theta)}_{m}\big)_{m\in \bZ}$ belongs to $\Delta$, and Propositions \ref{pr3} and \ref{pr4} associate to it an \emph{generalized $(n+1)$-system} $\uP^{(\theta)}$ on $(0,\infty)$, and a point $\xi^{(\theta)}\in\bR^{n+1}$. Extending the notation in an obvious manner gives
\begin{align*}
 r^{(\theta)}_{m}&=k a^{(\theta)}_{m}+(n-k+1) a^{(\theta)}_{m+1} < (n+1) a^{(\theta)}_{m+1}\\
t^{(\theta)}_{m}&= (k-1) a^{(\theta)}_{m} +a^{(\theta)}_{m+1}+ (n-k+1) a^{(\theta)}_{m+2} > a^{(\theta)}_{m+2}
\end{align*}
for all $m \in \bZ$, and $t^{(\theta)}_{m}/r^{(\theta)}_{m}$ tends to infinity with $m$. Thus, if $\theta,\theta' \in (0,\infty)$ with $\theta< \theta'$, then 
\[
 r^{(\theta)}_{m}< r^{(\theta')}_{m}= (\theta'/\theta ) r^{(\theta)}_{m} <  t^{(\theta)}_{m},
\]
for all sufficiently large $m \in \bZ$, and so
\begin{equation*}
\label{m5}
 \Vert\uP^{(\theta')}(r^{(\theta')}_{m})-\uP^{(\theta)}(r^{(\theta')}_{m})\Vert 
 \geqslant |P^{(\theta')}_{1}(r^{(\theta')}_{m})-P^{(\theta)}_{1}(r^{(\theta')}_{m})|=|a^{(\theta')}_{m} -a^{(\theta)}_{m}|=
(\theta' -\theta) 2^{m^3}.
\end{equation*}
This means that the difference $\uP^{(\theta')}-\uP^{(\theta)}$ is unbounded. Thus, the points $\xi^{(\theta')}$ and $\xi^{(\theta)}$ are distinct, and consequently, the map $\theta \mapsto \xi^{(\theta)}$ is injective on $(0,\infty)$. Its image is therefore uncountable.

\section{Proof of Theorem 3} 

Let $\xi$ be a point in $\bR^{n+1}$ whose coordinates are linearly independent over $\bQ$. On the model of Schmidt and Summerer in \cite[$\mathsection 1$]{S2009}, we define 
\begin{equation*}
\underline{\varphi}_{j}(\xi)=\liminf_{q\rightarrow \infty}\frac{L_{\xi,j}(q)}{q} \quad\text{and} \quad 
\overline{\varphi}_{j}(\xi)=\limsup_{q\rightarrow \infty}\frac{L_{\xi,j}(q)}{q} \qquad (1 \leqslant j\leqslant n+1).
\end{equation*} 
In \cite[$\mathsection 1$]{S2009}, Schmidt and Summerer show that these quantities satisfy
\begin{equation}
\label{p3}
\underline{\varphi}_{j+1}(\xi)\leqslant \overline{\varphi}_{j}(\xi) \qquad (1\leqslant j \leqslant n) .
\end{equation}

Now, suppose that $\overline{\varphi}_{k-1}(\xi)=0$ for some integer $k$ with $2\leqslant k\leqslant n$. Since $q- \sum\limits_{j=1}^{n+1} L_{\xi,j}(q)$ is a bounded function in $q$ on $(0,\infty)$, we have that
\begin{align*}
(n-k+2)\overline{\varphi}_{k}(\xi) &\leqslant \underset{q\rightarrow\infty}{\limsup}\,\, \frac{1}{q} \sum\limits_{j=k}^{n+1} L_{\xi,j}(q)
= \underset{q\rightarrow\infty}{\limsup} \,\, \frac{1}{q} \left( q - \sum\limits_{j=1}^{k-1} L_{\xi,j}(q) \right)=1 ,
\end{align*}
and so $\overline{\varphi}_{k}(\xi) \leqslant 1/(n-k+2)$. This yields 
 $\underline{\varphi}_{k+1}(\xi)\leqslant 1/(n-k+2)$.

\end{document}